\documentclass[letterpaper,11pt,twoside,keywordsasfootnote,addressatend,noinfoline]{article}
\usepackage{fullpage}
\usepackage[english]{babel}
\usepackage{amssymb}
\usepackage{amsmath}
\usepackage{theorem}
\usepackage{epsfig}
\usepackage{subfigure}
\usepackage{lscape}
\usepackage{imsart}
\usepackage{color}

\theorembodyfont{\slshape}
\newtheorem{theor}{Theorem}

\newtheorem{lemma}[theor]{Lemma}

\newenvironment{proof}{\noindent{\scshape Proof.}}{\hspace{2mm} $\square$}

\newcommand{\Z}{\mathbb{Z}}
\newcommand{\R}{\mathbb{R}}

\newcommand{\ep}{\epsilon}

\DeclareMathOperator{\sgn}{sgn \,}
\DeclareMathOperator{\poisson}{Poisson \,}
\DeclareMathOperator{\card}{card \,}

\begin{document}

\begin{frontmatter}

\title     {The naming game in language \\ dynamics revisited}
\runtitle  {The naming game in language dynamics revisited}
\author    {N. Lanchier\thanks{Research supported in part by NSF Grant DMS-10-05282.}}
\runauthor {N. Lanchier}
\address   {School of Mathematical and Statistical Sciences, \\ Arizona State University, \\ Tempe, AZ 85287, USA.}

\begin{abstract} \ \
 This article studies a biased version of the naming game in which players located on a connected graph interact through successive
 conversations to bootstrap a common name for a given object.
 Initially, all the players use the same word $B$ except for one bilingual individual who also uses word $A$.
 Both words are attributed a fitness, which measures how often players speak depending on the words they use and how often each word
 is pronounced by bilingual individuals.
 The limiting behavior depends on a single parameter: $\phi$ = the ratio of the fitness of word $A$ to the fitness of word $B$.
 The main objective is to determine whether word $A$ can invade the system and become the new linguistic convention.
 In the mean-field approximation, invasion of word $A$ is successful if and only if $\phi > 3$, a result that we also prove for the
 process on complete graphs relying on the optimal stopping theorem for supermartingales and random walk estimates.
 In contrast, for the process on the one-dimensional lattice, word $A$ can invade the system whenever $\phi > 1.053$ indicating that
 the probability of invasion and the critical value for $\phi$ strongly depend on the degree of the graph.
 The system on regular lattices in higher dimensions is also studied by comparing the process with percolation models.
\end{abstract}



\end{frontmatter}


\section{Introduction}
\label{sec:intro}

\indent The naming game was first proposed by Stells \cite{steels_1995} to describe the emergence of conventions and shared lexicons
 in a population of individuals interacting through successive conversations, but a number of variants of the model have also been
 introduced and studied numerically by statistical physicists, and we refer to Section V.B of \cite{castellano_2009} for a review of
 these different variants.
 The reason for the popularity of the naming game in the physics literature is that it is similar mathematically to traditional models
 in the field of statistical mechanics.
 The model studied in this paper is a biased version of the spatial naming game considered by Baronchelli et al. \cite{baronchelli_2006}.
 Their system consists of a population of individuals located on the vertex set of a finite connected graph that has to be thought of
 as an interaction network.
 Each individual is characterized by an internal inventory of words that are synonyms describing the same object.
 All inventories are initially empty and evolve through successive conversations:
 at each time step, an edge of the network is chosen uniformly at random, which causes the two individuals connected by the edge
 to interact.
 One individual is chosen at random to be the ``speaker'' making the other individual the ``hearer''.
 If the speaker does not have any word to describe the object then she invents one, whereas if she already has some words then she
 chooses one at random to be passed to the hearer.
 The conversation results in the following alternative:
 if the hearer already has the word pronounced in her internal inventory then this word is selected as the norm by both individuals --
 all the other words are removed from both inventories -- otherwise the hearer adds the word pronounced to her inventory.

\indent Based on numerical simulations, Baronchelli et al. \cite{baronchelli_2006} studied the maximum number of words present in the
 system as well as the time to global consensus, i.e., the time until all inventories consist of the same single word.
 In contrast, we use the naming game to study whether a new word can spread into a population that is already using another word as a
 convention, i.e., we assume that initially all the inventories reduce to the same single word, say word $B$, except for one individual
 who also has another word in her inventory, say word $A$.
 Under the symmetric rules of the naming game, the probability that $A$ becomes eventually the new convention tends to zero as the
 population size goes up to infinity so we look at biased versions of the naming game in which each word is attributed a fitness.
 In our model, the fitness of each word measures the fitness of each individual, that is how likely they are selected as a speaker
 rather than hearer, and also how likely each word is selected to be pronounced by bilingual individuals, i.e., individuals who
 possess both words in their internal inventory.
 Another significant difference between this article and previous works about the naming game is that it provides a rigorous analysis
 of the model on both finite and infinite graphs rather than results based on numerical simulations which are unavoidably restricted
 to finite graphs.
 Also, we describe the dynamics in continuous time rather than discrete time, i.e, we assume that conversations occur at rate one
 along each edge of the graph, in order to have a model well defined on finite and infinite graphs.

\indent To describe our biased version of the naming game more formally, we let $\phi_A$ and $\phi_B$ denote the fitness of word
 $A$ and word $B$, respectively, and set
 $$ \phi_{AB} \ := \ (1/2) \,(\phi_A + \phi_B) \quad \hbox{and} \quad p_{X \to Y} \ := \ \phi_X \,(\phi_X + \phi_Y)^{-1} $$
 for all $X, Y \in \{A, B, AB \}$.
 Note in particular that
 $$ p_{X \to X} \ = \ 1/2 \quad \hbox{and} \quad p_{X \to Y} + p_{Y \to X} \ = \ 1. $$
 The average fitness $\phi_{AB}$ represents the fitness of bilingual individuals.
 In each interaction, the individual playing the role of the speaker is chosen at random with probability her fitness divided by the
 overall fitness of the pair: when the neighbors are in state $X$ and $Y$, the individual chosen to be the speaker is the individual
 in state $X$ with probability $p_{X \to Y}$.
 Similarly, given that a bilingual individual is chosen as the speaker, the conditional probability that word $A$ is pronounced
 is equal to the relative fitness $p_{A \to B}$.
 In particular, each edge becomes active at rate one, which results in the following possible transitions for the states at the
 vertices connected by the edge:
\begin{equation}
\label{eq:naming-game}
 \begin{array}{rclcl}
   (A, B)   & \to & (A, AB)  & \hbox{with probability} & p_{A \to B} \vspace*{0pt} \\
            & \to & (AB, B)  & \hbox{with probability} & p_{B \to A} \vspace*{6pt} \\
   (A, AB)  & \to & (A, A)   & \hbox{with probability} & p_{A \to AB} + p_{AB \to A} \ p_{A \to B} \vspace*{0pt} \\
            & \to & (AB, AB) & \hbox{with probability} & p_{AB \to A} \ p_{B \to A} \vspace*{6pt} \\
   (B, AB)  & \to & (B, B)   & \hbox{with probability} & p_{B \to AB} + p_{AB \to B} \ p_{B \to A} \vspace*{0pt} \\
            & \to & (AB, AB) & \hbox{with probability} & p_{AB \to B} \ p_{A \to B} \vspace*{6pt} \\
   (AB, AB) & \to & (A, A)   & \hbox{with probability} & p_{A \to B} \vspace*{0pt} \\
            & \to & (B, B)   & \hbox{with probability} & p_{B \to A} \end{array}
\end{equation}
 Note that, when the fitnesses are equal, one recovers the transition probabilities of the unbiased naming game described above.
 We formulate the dynamics using two parameters to have natural notations that preserve the symmetry between both words, but we
 point out that the long-term behavior of the process only depends on the ratio $\phi := \phi_A / \phi_B$. \vspace*{8pt}

\newpage


\noindent {\bf Mean-field model.}
\begin{figure}[t!]
\centering
 \includegraphics[width = 440pt]{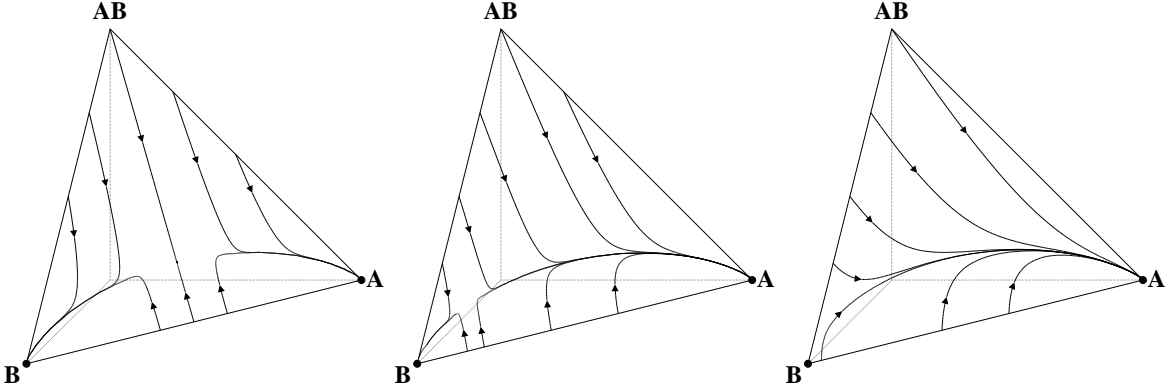}
\caption{\upshape Solution curves of the mean-field model with $\phi = 1$, $\phi = 1.5$ and $\phi = 4$, respectively.}
\label{fig:mean-field}
\end{figure}
 Before stating our results for the spatial stochastic model, we look at its nonspatial deterministic mean-field approximation,
 i.e, the model obtained by assuming that the population is well-mixing.
 This results in the following system of differential equations:
 $$ \begin{array}{rcl}
         u_A' & = & u_A \,u_{AB} \,(1 - 2 \,p_{AB \to A} \ p_{B \to A}) \ - \ u_A \,u_B \ p_{B \to A} \ + \ 2 \,u_{AB}^2 \ p_{A \to B} \vspace*{4pt} \\
         u_B' & = & u_B \,u_{AB} \,(1 - 2 \,p_{AB \to B} \ p_{A \to B}) \ - \ u_A \,u_B \ p_{A \to B} \ + \ 2 \,u_{AB}^2 \ p_{B \to A} \vspace*{4pt} \\
      u_{AB}' & = & - (u_A + u_B)' \end{array} $$
 where $u_X$ denotes the frequency of type $X$ individuals for $X \in \{A, B, AB \}$.
 The mean-field model has two trivial equilibria, namely
 $$ e_A \ := \ (1, 0, 0) \qquad \hbox{and} \qquad e_B \ := \ (0, 1, 0) $$
 which correspond to the configuration in which all individuals are of type $A$ and the configuration in which all individuals
 are of type $B$, respectively.
 We say that word $A$ can invade word $B$ in the mean-field model whenever the system starting from any initial state different
 from $e_B$ converges to the trivial equilibrium $e_A$.
 Regardless of the ratio $\phi := \phi_A / \phi_B$, the frequency of type $A$ individuals might decrease because the boundary
 $u_{AB} = 0$ is repelling, but looking instead at the difference between the frequency of individuals using word $A$ and
 word $B$ gives
 $$ \begin{array}{rcl}
     (u_A - u_B)' & = & u_A \,u_{AB} \,(1 - 2 \,p_{AB \to A} \ p_{B \to A}) \vspace*{4pt} \\ && \hspace*{10pt} - \
      u_B \,u_{AB} \,(1 - 2 \,p_{AB \to B} \ p_{A \to B}) \ + \ (u_A \,u_B + 2 \,u_{AB}^2)(p_{A \to B} - p_{B \to A}) \vspace*{4pt} \\ & = &
     (3 \phi - 1)(3 \phi + 1)^{-1} \,u_A \,u_{AB} \vspace*{4pt} \\ && \hspace*{10pt} + \
     (\phi - 3)(\phi + 3)^{-1} \,u_B \,u_{AB} \ + \ (\phi - 1)(\phi + 1)^{-1} \,(u_A \,u_B + 2 \,u_{AB}^2), \end{array} $$
 which is positive for all $\phi > 3$ when $u_A \neq 1$ and $u_B \neq 1$.
 This implies that there is no equilibrium other than the two trivial equilibria and that word $A$ can invade word $B$ for all $\phi > 3$.
 This condition is sharp in the sense that $e_B$ is locally stable when $\phi < 3$.
 Indeed, the Jacobian matrix of the system of differential equations at point $e_B$ reduces to
 $$ \mathcal J_{e_B} \ = \ \left(\begin{array}{ccccc}
  - p_{B \to A} &&  0 && 0 \vspace*{5pt} \\
  - p_{A \to B} &&  0 && 1 - 2 q_A \vspace*{5pt} \\
    1           &&  0 && 2 q_A - 1 \end{array} \right) $$
 where $q_A := p_{AB \to B} \ p_{A \to B}$ is a key quantity that will appear again later.
 The eigenspace associated with the eigenvalue zero is generated by the vector $(0, 1, 0)$ which is not oriented in the direction of the
 two-simplex containing the solution curves.
 The other two eigenvalues are
 $$ - p_{B \to A} \ = \ - (\phi + 1)^{-1} \ < \ 0 \qquad \hbox{and} \qquad 2 q_A - 1 \ = \ (\phi - 3)(\phi + 3)^{-1} $$
 which are both negative when $\phi < 3$.
 In particular, for all $\phi < 3$, the equilibrium $e_B$ is locally stable, therefore word $A$ cannot invade.
 Note that the obvious symmetry of the model also implies that both trivial equilibria are locally stable when $1/3 < \phi < 3$.
 Numerical simulations of the mean-field model suggest that, in this case, there is an additional nontrivial fixed point which
 is a saddle point, therefore the system is bistable:
 for almost all initial conditions, the system converges to one of the two trivial equilibria (see Figure \ref{fig:mean-field} for
 pictures of the solution curves). \vspace*{8pt}


\noindent {\bf Spatial stochastic model.}
 We now look at the spatial stochastic naming game \eqref{eq:naming-game}.
 For the stochastic process, the main objective is to study the probability that word $A$ invades the population and is selected as a
 new linguistic convention when starting with a single bilingual individual and all the other individuals of type $B$.
 Note that, for non-homogeneous graphs, this probability depends on the location of the initial bilingual individual.
 Also, letting $\eta_t (x)$ be the state of the individual at vertex $x$ at time $t$, and letting $P_x$ denote the law of the process
 starting with
 $$ \eta_0 (x) = AB \quad \hbox{and} \quad \eta_0 (y) = B \ \ \hbox{for all} \ y \in V, \ y \neq x $$
 we define the probability of invasion as
\begin{equation}
\label{eq:invasion}
 \begin{array}{l}
  p_A \ := \ \inf_{x \in V} \,P_x \,(\lim_{t \to \infty} \eta_t (y) = A \ \hbox{for all} \ y \in V). \end{array}
\end{equation}
 Interestingly, our results indicate that the probability of invasion strongly depends on the topology of the network of interactions,
 suggesting that, on regular graphs, it is decreasing with respect to the degree of the network, a property that cannot be captured by
 the mean-field model since it excludes any spatial structure.
 To begin with, we look at finite graphs.
 Our first theorem extends the first result found for the mean-field model: word $A$ can invade word $B$ for all $\phi > 3$.
\begin{theor} --
\label{th:finite-graphs}
 Assume that $G$ is finite and $\phi > 3$. Then, $p_A \geq 1 - 3 / \phi > 0$.
\end{theor}
 Note that on finite graphs $p_A$ is always positive but might vanish to zero as the population size increases.
 In contrast, Theorem \ref{th:finite-graphs} shows more particularly that the probability of invasion is bounded from below by a
 constant that depends on the ratio $\phi$ but not on the number of vertices.
 The idea of the proof is to show first that a certain function of the number of type $A$ individuals and the number of type $B$
 individuals is a supermartingale with respect to the $\sigma$-algebra generated by the process and then apply the optimal stopping
 theorem.
 Our next result indicates that the invadability condition in Theorem \ref{th:finite-graphs} is sharp for complete graphs in the sense
 that the probability of invasion vanishes to zero as the population increases when $\phi < 3$.
\begin{theor} --
\label{th:complete-graphs}
 Assume that $G$ is the complete graph with $N$ vertices. Then,
 $$ \lim_{N \to \infty} \ p_A \ = \ 0 \quad \hbox{for all} \ \ \phi < 3. $$
\end{theor}
 In the proof of Theorem \ref{th:finite-graphs}, the dynamics of the number of type $A$ and type $B$ is expressed as a function of the
 number of edges of different types.
 The complete graph is the only graph for which the number of edges of different types can be expressed as a function of the number of
 individuals of different types.
 Also, one of the keys to proving the theorem is to use the fact that, on complete graphs, the number of individuals in different states
 becomes a Markov chain.
 The combination of both theorems indicates that the dynamics of the naming game on complete graphs is well captured by
 the mean-field approximation.
 Our next result shows more interestingly that this is not true for the process on the infinite one-dimensional lattice,
 suggesting that the critical value for the ratio of the fitnesses decreases as the degree of the graph decreases.
\begin{theor} --
\label{th:1D}
 In one dimension, $p_A > 0$ whenever $\phi > c$ where
 $$ c \ := \ \frac{23 + \sqrt{6097}}{96} \ \approx \ 1.053 \quad \hbox{satisfies} \quad 48 c^2 - 23 c - 29 = 0. $$
\end{theor}
 The proof of Theorem \ref{th:1D} is based on the analysis of the interface between individuals in different states, which is only
 possible in one dimension.
 The bound $c$ is not sharp but our approach to prove the theorem together with the obvious symmetry of the model implies that the
 critical ratio is between $c^{-1}$ and $c$, which suggests that the critical ratio is equal to one: the probability of a successful
 invasion is positive if and only if $\phi > 1$.
 Finally, we look at the naming game on regular lattices in higher dimensions.
 In this case, using a block construction to compare the process properly rescaled in space and time with oriented site percolation,
 it can be proved that the probability of invasion is positive for $\phi$ sufficiently large.
\begin{theor} --
\label{th:lattice}
 In any dimension, $p_A > 0$ whenever $\phi$ is large enough.
\end{theor}
 Our approach can be improved to get an explicit bound for the critical value for $\phi$ but this bound is far from being optimal.
 We conjecture as in one dimension that the critical ratio is equal to one, which is supported by numerical simulations of the process.
 More generally, we conjecture that, on connected graphs in which the degree is uniformly bounded by a fixed constant $K$, the
 critical value is equal to one in the sense that the probability of invasion is bounded from below by a positive constant that
 only depends on $K$, in disagreement with the mean-field model.


\section{Preliminary results}
\label{sec:preliminary}

\begin{table}[t!]
\begin{center}
\begin{tabular}{rlll}
    & transitions for $(\zeta_t)$ & condition on $U_n (x, y)$  & possible transitions for $(\xi_t)$ \vspace*{6pt} \\
 1A & $(A, A) \to (A, A)$         & none                       & any \vspace*{6pt} \\
 2A & $(A, AB) \to (A, A)$        & $U_n (x, y) < 1 - q_B$     & 2A, 3A, 3B, 4A, 4B, 5A, 5B, 6B \vspace*{2pt} \\
 2B & $(A, AB) \to (AB, AB)$      & $U_n (x, y) > 1 - q_B$     & 2B, 3B, 4B, 5A, 5B, 6B (excludes 3A, 4A) \vspace*{6pt} \\
 3A & $(A, B) \to (A, AB)$        & $U_n (x, y) < p_{A \to B}$ & 3A, 5A, 5B, 6B \vspace*{2pt} \\
 3B & $(A, B) \to (AB, B)$        & $U_n (x, y) > p_{A \to B}$ & 3B, 5B, 6B (excludes 5A) \vspace*{6pt} \\
 4A & $(AB, AB) \to (A, A)$       & $U_n (x, y) < p_{A \to B}$ & 4A, 5A, 5B, 6B \vspace*{2pt} \\
 4B & $(AB, AB) \to (B, B)$       & $U_n (x, y) > p_{A \to B}$ & 4B, 5B, 6B (excludes 5A) \vspace*{6pt} \\
 5A & $(AB, B) \to (AB, AB)$      & $U_n (x, y) < q_A$         & 5A, 6B \vspace*{2pt} \\
 5B & $(AB, B) \to (B, B)$        & $U_n (x, y) > q_A$         & 5B, 6B \vspace*{6pt} \\
 6B & $(B, B) \to (B, B)$         & none                       & only 6B
\end{tabular}
\end{center}
\caption{\upshape{Coupling between the processes $(\zeta_t)$ and $(\xi_t)$.}}
\label{tab:coupling}
\end{table}

\indent In this section, we state some basic properties about the naming game that will be useful in the subsequent sections.
 A common aspect of all our proofs is to think of the process as being constructed graphically from independent Poisson processes that
 indicate the time of the interactions, a popular idea in the field of interacting particle systems due to Harris \cite{harris_1972}.
 In the case of the naming game, additional collections of uniform random variables must be introduced to
 also indicate the outcome of each interaction.
 More precisely, for every edge $(x, y) \in E$, we let
\begin{itemize}
 \item $\{T_n (x, y) : n \geq 1 \}$ be the arrival times of a rate one Poisson process, and \vspace*{4pt}
 \item $\{U_n (x, y) : n \geq 1 \}$ be independent uniform random variables over $(0, 1)$.
\end{itemize}
 Collections of random variables attached to different edges are also independent.
 The process is then constructed as follows:
 at time $T_n (x, y)$, the states at $x$ and $y$ are simultaneously updated according to the transitions in the left column
 of Table \ref{tab:coupling}.
 Since interactions involving both words can each result in two different outcomes depending on whether word $A$ or word $B$
 is pronounced, the random variable $U_n (x, y)$ is used to account for the probability of each outcome as indicated by the
 conditions in the middle column of the table where
\begin{equation}
\label{eq:short}
  q_A \ := \ p_{AB \to B} \ p_{A \to B} \quad \hbox{and} \quad q_B \ := \ p_{AB \to A} \ p_{B \to A}.
\end{equation}
 Note that $q_A$ is the probability that word $A$ is pronounced in a conversation involving a bilingual individual and a type
 $B$ individual.
 One can easily check that the conditions in the table indeed produce the desired transition probabilities in \eqref{eq:naming-game}.
 Based on this graphical representation, processes with different parameters or starting from different initial configurations
 can be coupled to prove important monotonicity results.
 The next lemma shows for instance a certain monotonicity of the naming game with respect to its initial configuration, which can
 be viewed as the analog of attractiveness for spin systems.
 This result will be useful in the proof of Theorem \ref{th:1D}.

\begin{lemma} --
\label{lem:attractive}
 Let $(\zeta_t)$ and $(\xi_t)$ be two copies of the naming game. Then,
 $$ P \,(\xi_t (x) = A) \leq P \,(\zeta_t (x) = A) \quad \hbox{and} \quad P \,(\xi_t (x) = B) \geq P \,(\zeta_t (x) = B) $$
 for all $(x, t) \in V \times (0, \infty)$ provided this holds for all $(x, t) \in V \times \{0 \}$.
\end{lemma}
\begin{proof}
 The result follows from a coupling of the two processes that we construct conjointly from the same graphical representation.
 That is, we assume that
 $$ (\xi_0 (z) = A \ \hbox{implies} \ \zeta_0 (z) = A) \quad \hbox{and} \quad
    (\zeta_0 (z) = B \ \hbox{implies} \ \xi_0 (z) = B) \quad \hbox{for all} \ z \in V $$
 and that both processes are constructed from the same Poisson processes and the same collections of uniform random variables.
 The construction given by Harris \cite{harris_1972}, which relies on arguments from percolation theory, implies that, for any
 small enough time interval, there exists a partition of the vertex set into almost surely finite connected components such that
 any two vertices in two different components do not influence each other in the time interval.
 Since the number of interactions in each component in the time interval is almost surely finite, the result can be proved for
 each of these finite space-time regions by induction. Assume that
 $$ (\xi_{t-} (z) = A \ \hbox{implies} \ \zeta_{t-} (z) = A) \quad \hbox{and} \quad
    (\zeta_{t-} (z) = B \ \hbox{implies} \ \xi_{t-} (z) = B) \quad \hbox{for all} \ z \in V $$
 for some arrival time $t := T_n (x, y)$.
 To prove that the previous relationship between both processes is preserved at time $t$, we observe that the interaction between
 the individuals at $x$ and $y$ can result in ten different transitions depending on the state of both individuals.
 These transitions are listed in the left column of Table \ref{tab:coupling} and can be divided into two types:
\begin{itemize}
 \item the transitions that create an $A$ or remove a $B$, which are labeled 2A--5A, \vspace*{2pt}
 \item the transitions that create a $B$ or remove an $A$, which are labeled 2B--5B.
\end{itemize}
 As previously mentioned, except for transitions 1A and 6B, every other pair of states for the neighbors can result in two
 possible transitions depending on whether word $A$ or word $B$ is pronounced during the conversation.
 The last column of the table indicates that for all possible simultaneous updates of both processes, the ordering between both
 processes is preserved at time $t$, i.e.,
 $$ (\xi_t (z) = A \ \hbox{implies} \ \zeta_t (z) = A) \quad \hbox{and} \quad
    (\zeta_t (z) = B \ \hbox{implies} \ \xi_t (z) = B) \quad \hbox{for all} \ z \in V. $$
 To prove, as indicated in the last column, that a transition 2B in the first process indeed excludes the transitions 3A and 4A
 in the second process, we observe that
 $$ \begin{array}{rcl}
     1 - q_B & = & p_{A \to B} + p_{B \to A} - p_{AB \to A} \ p_{B \to A} \vspace*{4pt} \\ & = &
     p_{A \to B} + p_{B \to A} \ (1 - p_{AB \to A}) \ \geq \ p_{A \to B} \end{array} $$
 which gives the implication
\begin{equation}
\label{eq:attractive-1}
  U_n (x, y) > 1 - q_B \quad \hbox{implies that} \quad U_n (x, y) > p_{A \to B}
\end{equation}
 and proves the exclusion of type 3A and 4A transitions. Similarly,
\begin{equation}
\label{eq:attractive-2}
  U_n (x, y) > p_{A \to B} \quad \hbox{implies that} \quad U_n (x, y) > p_{AB \to B} \ p_{A \to B} = q_A
\end{equation}
 showing that the transitions 3B and 4B in the first process exclude transition 5A in the second process.
 As previously mentioned, the lemma follows from the fact that all possible simultaneous updates of both processes given in the
 last column preserve the desired ordering.
\end{proof}


\section{The naming game on finite graphs}
\label{sec:finite}

\indent This section is devoted to the proofs of Theorem \ref{th:finite-graphs} and Theorem \ref{th:complete-graphs} about
 the naming game on finite connected graphs.
 The key to proving the first theorem is to show that a certain process that depends on the difference between the number
 of individuals using word $A$ and the number of individuals using word $B$ is a supermartingale with respect to the natural
 filtration of the naming game, which allows to directly deduce the theorem from the optimal stopping theorem.
 To prove the second theorem which specializes in the process on complete graphs, the idea is to observe that, as long as
 bilingual individuals do not interact with each other, there is no type $A$ individual in the population and the number
 of bilingual individuals evolves like a subcritical birth and death process that goes extinct quickly.
 Throughout this section, $A_t$ and $B_t$ denote respectively the number of individuals of type $A$ and type $B$ at time $t$,
 and we let
 $$ \begin{array}{rcl}
     e_t (X, Y) & := & \hbox{number of edges connecting a type $X$ individual} \vspace*{2pt} \\
                &    & \hbox{and a type $Y$ individual at time $t$} \end{array} $$
 for all $X, Y \in \{A, B, AB \}$.
 To motivate our proof of the first theorem and explain the assumption, we observe that the transitions labeled 2A--5A
 in Table \ref{tab:coupling}, which are the transitions that increase the number of individuals using $A$ or decrease the
 number of individuals using $B$, all occur with probability at least one half if and only if $\phi > 3$.
 As shown in the next lemma, this property can be used to construct a certain supermartingale with respect to the natural
 filtration of the process: the $\sigma$-algebra $\mathcal F_t$ generated by the realization of the naming game until time $t$.

\begin{lemma} --
\label{lem:supermartingale}
 Assume that $\phi \geq 3$.
 Then, for all $s > t$,
 $$ E \,(M_s \,| \,\mathcal F_t) \leq M_t \quad \hbox{where} \quad M_t := a^{A_t - B_t} \ \ \hbox{and} \ \ a := 3 / \phi. $$
\end{lemma}
\begin{proof}
 Using the transition probabilities in Table \ref{tab:coupling}, we get
 $$ \begin{array}{l}
    \lim_{h \to 0} \ h^{-1} \ E \,(M_{t + h} - M_t \,| \,\mathcal F_t) \ = \
    \sum_{j = -2}^2 \ (a^j - 1) \,M_t \ \lim_{h \to 0} \ h^{-1} \ P \,(M_{t + h} = M_t + j \,| \,\mathcal F_t) \vspace*{4pt} \\ \hspace*{20pt} = \
     (a - 1) \,M_t \,(e_t (A, AB) \,(1 - q_B) + e_t (A, B) \,p_{A \to B} + e_t (B, AB) \,q_A) \vspace*{4pt} \\ \hspace*{40pt} + \
     (a^{-1} - 1) \,M_t \,(e_t (B, AB) \,(1 - q_A) + e_t (A, B) \,p_{B \to A} + e_t (A, AB) \,q_B) \vspace*{4pt} \\ \hspace*{40pt} + \
      M_t \,e_t (AB, AB) \,((a^2 - 1) \,p_{A \to B} + (a^{-2} - 1) \,p_{B \to A}). \end{array} $$
 Re-arranging the terms with respect to the type of edges, this becomes
\begin{equation}
\label{eq:martingale-0}
  \begin{array}{l}
    \lim_{h \to 0} \ h^{-1} \ E \,(M_{t + h} - M_t \,| \,\mathcal F_t) \ = \
     M_t \,e_t (A, AB) \,((a - 1) (1 - q_B) + (a^{-1} - 1) \,q_B) \vspace*{4pt} \\ \hspace*{120pt} + \
     M_t \,e_t (B, AB) \,((a - 1) \,q_A + (a^{-1} - 1) (1 - q_A)) \vspace*{4pt} \\ \hspace*{120pt} + \
     M_t \,e_t (A, B) \,((a - 1) \,p_{A \to B} + (a^{-1} - 1) \,p_{B \to A}) \vspace*{4pt} \\ \hspace*{120pt} + \
     M_t \,e_t (AB, AB) \,((a^2 - 1) \,p_{A \to B} + (a^{-2} - 1) \,p_{B \to A}). \end{array}
\end{equation}
 First, we observe that $q_B = (3 \phi + 1)^{-1}$ and, for all $\phi \geq 1/3$,
\begin{equation}
\label{eq:martingale-1}
  \begin{array}{l}
     (a - 1) (1 - q_B) + (a^{-1} - 1) \,q_B \ = \ a^{-1} \,(a - 1) ((1 - q_B) \,a - q_B) \vspace*{4pt} \\ \hspace*{60pt}
                                            \ = \ a^{-1} \,(3 \phi + 1)^{-1} \,(a - 1) (3 \phi \,a - 1) \ \leq \ 0 \ \ \hbox{for all} \ \ (3 \phi)^{-1} \leq a \leq 1. \end{array}
\end{equation}
 Similarly, $q_A = \phi \,(\phi + 3)^{-1}$ and, for all $\phi \geq 3$, we have
\begin{equation}
\label{eq:martingale-2}
  \begin{array}{l}
     (a - 1) \,q_A + (a^{-1} - 1) (1 - q_A) - 1 \ = \ a^{-1} \,(a - 1) (q_A \,a - (1 - q_A)) \vspace*{4pt} \\ \hspace*{60pt}
                                                \ = \ a^{-1} \,(\phi + 3)^{-1} \,(a - 1) (\phi \,a - 3) \ \leq \ 0 \ \ \hbox{for all} \ \ 3 \phi^{-1} \leq a \leq 1. \end{array}
\end{equation}
 Finally, $p_{A \to B} = \phi \,(\phi + 1)^{-1}$ and, for all $\phi \geq 1$, we have
\begin{equation}
\label{eq:martingale-3}
  \begin{array}{l}
     (a - 1) \,p_{A \to B} + (a^{-1} - 1) \,p_{B \to A} \ = \ a^{-1} \,(a - 1) (p_{A \to B} \,a - p_{B \to A}) \vspace*{4pt} \\ \hspace*{60pt}
                                                        \ = \ a^{-1} \,(\phi + 1)^{-1} \,(a - 1) (\phi \,a - 1) \ \leq \ 0 \ \ \hbox{for all} \ \ \phi^{-1} \leq a \leq 1 \end{array}
\end{equation}
 from which we also deduce that, for all $\phi \geq 1$,
\begin{equation}
\label{eq:martingale-4}
  (a - 1) \,p_{A \to B} + (a^{-1} - 1) \,p_{B \to A} \ \leq 0 \ \ \hbox{for all} \ \ 1/ \sqrt{\phi} \leq a \leq 1.
\end{equation}
 Plugging \eqref{eq:martingale-1}--\eqref{eq:martingale-4} into \eqref{eq:martingale-0}, we conclude that
 $$ \begin{array}{l} \lim_{h \to 0} \ h^{-1} \ E \,(M_{t + h} - M_t \,| \,\mathcal F_t) \ \leq \ 0 \quad \hbox{for all $\phi \geq 3$ and $a = 3 / \phi$} \end{array} $$
 showing that $(M_t)$ is a supermartingale for $a = 3 / \phi$.
\end{proof} \\ \\

\pagebreak

\noindent Applying the optimal stopping theorem to $(M_t)$ gives the following result.
\begin{lemma} --
\label{lem:stopping}
 For all $\phi > 3$, we have
 $$ \begin{array}{l}
     p_A \ := \ \inf_{x \in V} \,P_x \,(B_t = 0 \ \hbox{for some} \ t) \ \geq \ 1 - 3 / \phi \ > \ 0. \end{array} $$
\end{lemma}
\begin{proof}
 First, we introduce the stopping time
 $$ T \ := \ \inf \,\{t : A_t - B_t \in \{-N, N \} \} $$
 where $N$ denotes the number of vertices.
 Since the naming game on any finite graph converges almost surely to the configuration in which all individuals are monolingual of the
 same type, the stopping time $T$ is almost surely finite.
 Using in addition that the process $(M_t)$ is a supermartingale according to Lemma \ref{lem:supermartingale}, we deduce from the optimal
 stopping theorem that
 $$ E \,M_T \ = \ E \,(a^{X_T - Y_T}) \ \leq \ a^N \,p_A + a^{-N} \,(1 - p_A) \ \leq \ E \,M_0 \ = \ a^{- (N - 1)} $$
 for all $a = 3 / \phi < 1$. In particular,
 $$ \begin{array}{l}
     p_A  \ \geq \ (a^{- (N - 1)} - a^{-N})(a^N - a^{-N})^{-1} \vspace*{4pt} \\ \hspace*{50pt} = \
                 (1 - a)(1 - a^{2N})^{-1} \ \geq \ 1 - a \ = \ 1 - 3 / \phi \ > \ 0 \end{array} $$
 which completes the proof of the lemma.
\end{proof} \\ \\
 Theorem \ref{th:finite-graphs} directly follows from Lemma \ref{lem:stopping} by observing that the probability $p_A$ in the statement
 of the lemma is precisely the probability $p_A$ in the statement of the theorem.
 We now focus on the naming game on the complete graph.
 Note that in this case the number of edges of each type can be expressed as a function of the number of individuals of each type,
 therefore $(A_t, B_t)$ is now a continuous-time Markov chain.
 As previously mentioned, to prove that $p_A$ tends to zero as the number of vertices goes to infinity, the idea is to
 observe that, as long as bilingual individuals do not interact with each other, there is no type $A$ individual in the population
 and the number of bilingual individuals evolves like a subcritical birth and death process.
 To make the argument precise, we introduce the birth and death process $(Z_t)$ starting with a single individual and with birth
 rate $N q_A$ and death rate $N (1 - q_A)$, i.e.,
 $$ \begin{array}{rcll}
    \lim_{h \to 0} \ h^{-1} \ P \,(Z_{t + h} = j \,| \,X_t = i)
     & = & i \,N q_A & \hbox{for} \ \ j = i + 1 \vspace*{2pt} \\
     & = & i \,N (1 - q_A) & \hbox{for} \ \ j = i - 1. \end{array} $$
 We start with the following preliminary result about the number of jumps $J$ before extinction of subcritical birth and death processes.
\begin{lemma} --
\label{lem:birth-death}
 Fix $\phi < 3$ and $\ep > 0$, and let $J := \ \card \{t : Z_t \neq Z_{t-} \}$. Then,
 $$ P \,(J \geq 2n_{\ep} + 1 \,| \,Z_0 = 1) < \ep \quad \hbox{for all $n_{\ep}$ large}. $$
\end{lemma}
\begin{proof}
 First, we note that, since $\phi < 3$,
 $$ q_A \ = \ p_{AB \to B} \ p_{A \to B} \ = \ \frac{\phi_A}{2 \,(\phi_{AB} + \phi_B)} \ = \ \frac{\phi}{\phi + 3} \ < \ \frac{3}{\phi + 3} \ = \ 1 - q_A $$
 from which it follows that
 $$ P \,(J < \infty) \ = \ P \,(Z_t = 0 \ \hbox{for some} \ t \ | \ Z_0 = 1) \ = \ 1. $$
 Moreover, using that the number of paths of length $2n$ not crossing 0 is bounded by the total number of paths of length $2n$ together
 with Stirling's formula, we get
 $$ P \,(J = 2n + 1) \ \leq \ {2n \choose n} \ q_A^n \ (1 - q_A)^{n + 1} \ \leq \ \frac{(4 \,q_A \,(1 - q_A))^n}{\sqrt{\pi n}} $$
 for all $n$ large.
 In particular,
 $$ P \,(J \geq 2n_{\ep} + 1) \ \leq \ P \,(J = \infty) \ + \ \sum_{n = n_{\ep}}^{\infty} \ \frac{(4 \,q_A \,(1 - q_A))^n}{\sqrt{\pi n}} \ < \ \ep $$
 for all $n_{\ep}$ large since $4 \,q_A (1 - q_A) < 1$.
\end{proof} \\ \\
 The reason for introducing the birth and death process above is that the number of bilingual individuals evolves precisely according
 to this process until two bilingual individuals interact with each other, an event that we call a collision.
 In particular, it can be deduced from the previous lemma that the probability that a collision ever happens is small when $N$ is large,
 which is also a bound for the probability that word $A$ outcompetes word $B$.
 To prove this result, we let
 $$ \tau_C \ := \ \inf \,\{t : t = T_n (x, y) \ \hbox{for some} \ x, y \in V \ \hbox{with} \ \eta_{t-} (x) = \eta_{t-} (y) = AB \}. $$
 be the time of the first collision.
\begin{lemma} --
\label{lem:collision}
 Fix $\phi < 3$ and $\ep > 0$. Then,
 $$ P \,(\tau_C < \infty \,| \,A_0 = 0 \ \hbox{and} \ B_0 = N - 1) < 2 \ep \quad \hbox{for all $N$ large}. $$
\end{lemma}
\begin{proof}
 To begin with, we observe that, before the time $\tau_C$ of the first collision, there is no monolingual individual of type $A$ in
 the population.
 In particular, using the expression of the transition probabilities in the second column of Table \ref{tab:coupling}, and introducing
 $$ \begin{array}{l}
     r (i, j) \ := \ \lim_{h \to 0} \ h^{-1} \ P \,(A_{t + h} = A_t + i \ \hbox{and} \ B_{t + h} = B_t + j \ | \ \mathcal F_t), \end{array} $$
 we obtain that, before the first collision,
\begin{equation}
\label{eq:rates}
 \begin{array}{rclrcl}
   r (0, -1) & = & q_A \ e_t (B, AB)       & \hspace*{25pt} r (+2, 0) & = & p_{A \to B} \ e_t (AB, AB) \vspace*{2pt} \\
   r (0, +1) & = & (1 - q_A) \ e_t (B, AB) & \hspace*{25pt} r (0, +2) & = & p_{B \to A} \ e_t (AB, AB) \end{array}
\end{equation}
 whereas $r (i, j) = 0$ for all other values of $i$ and $j$.
 This implies that, before the first collision, the number of bilingual individuals has evolved according to the birth and death
 process in which individuals independently give birth at rate $N q_A$ and die at rate $N (1 - q_A)$.
 In particular, the naming game can be coupled with the birth and death process in such a way that
 $$ P \,(A_t = 0 \ \hbox{and} \ (AB)_t = Z_t \ | \ \tau_C > t) \ = \ 1 $$
 where $(AB)_t$ denotes the number of bilingual individuals at time $t$.
 The rest of the proof relies on the fact that the probability that the number of jumps in the birth and death process is large and
 the probability that there is a collision in the naming game coupled with the birth and death process when the number of jumps
 is small are both small when the graph is large.
 Indeed, Lemma~\ref{lem:birth-death} gives the existence of $n_{\ep}$ fixed from now on such that
\begin{equation}
\label{eq:collision-1}
 \sum_{n \geq n_{\ep}} \ P \,(\tau_C < \infty \,| \,J = 2n + 1) \ P \,(J = 2n + 1) \ \leq \ P \,(J \geq 2 n_{\ep} + 1) \ < \ \ep.
\end{equation}
 Moreover, when $J = 2n + 1$, the maximum number of individuals cannot exceed $n + 1$ in the birth and death process therefore, thinking
 again of the number of bilingual individuals as being coupled with the birth and death process before the first collision, at each jump,
 the probability of a collision is bounded by $N^{-1} (n + 1)$.
 The integer $n_{\ep}$ being fixed, this implies that
\begin{equation}
\label{eq:collision-2}
 \begin{array}{l}
  \displaystyle \sum_{n < n_{\ep}} \ P \,(\tau_C < \infty \,| \,J = 2n + 1) \ P \,(J = 2n + 1) \vspace*{-4pt} \\ \hspace*{40pt} \ \leq \
  \displaystyle \sum_{n < n_{\ep}} \ P \,(\tau_C < \infty \,| \,J = 2n + 1) \ \leq \
  \displaystyle \sum_{n < n_{\ep}} \ N^{-1} \,(2n + 1)(n + 1) \ < \ \ep \end{array}
\end{equation}
 for all $N$ sufficiently large.
 The lemma simply follows by observing that the probability to be estimated is bounded by the sum of the probabilities
 in \eqref{eq:collision-1} and \eqref{eq:collision-2}.
\end{proof} \\ \\
 Theorem \ref{th:complete-graphs} directly follows from the next lemma.
\begin{lemma} --
\label{lem:extinction}
 Fix $\phi < 3$ and $\ep > 0$. Then,
 $$ P \,(\eta_t \equiv A \ \hbox{for some} \ t \ | \,A_0 = 0 \ \hbox{and} \ (AB)_0 = 1) < 2 \ep \quad \hbox{for all $N$ large}. $$
\end{lemma}
\begin{proof}
 Since there is no type $A$ individual before the first collision,
 $$ \begin{array}{l}
     P \,(\eta_t \equiv A \ \hbox{for some} \ t \ | \,A_0 = 0 \ \hbox{and} \ (AB)_0 = 1) \vspace*{4pt} \\ \hspace*{50pt} \leq \
     P \,(\eta_t (x) = A \ \hbox{for some} \ (x, t) \in V \times \R_+ \,| \,A_0 = 0 \ \hbox{and} \ (AB)_0 = 1) \vspace*{4pt} \\ \hspace*{50pt} \leq \
     P \,(\tau_C < \infty \,| \,A_0 = 0 \ \hbox{and} \ (AB)_0 = 1) \ < \ 2 \ep \end{array} $$
 for all $N$ sufficiently large according to Lemma \ref{lem:collision}.
\end{proof}


\section{The naming game in one dimension}
\label{sec:1D}

\indent This section is devoted to the proof of Theorem \ref{th:1D}.
 The first and main step of the proof is to show almost sure invasion of word $A$ for the naming game $(\zeta_t)$ starting with
\begin{equation}
\label{eq:initial}
 \zeta_0 (x) = A \ \ \hbox{for all} \ x \leq 0 \quad \hbox{and} \quad \zeta_0 (x) = B \ \ \hbox{for all} \ x > 0.
\end{equation}
 The main difficulty to prove this result is that, even in the presence of nearest neighbor interactions, the evolution rules
 in \eqref{eq:naming-game} can create infinitely many interfaces, i.e., the state space of the process seen from the rightmost
 type $A$ individual that has only type $A$ to her left is infinite.
 Motivated by numerical simulations of the process that suggest that the size of the interface is somewhat small most of the
 time, we prove the result for the process $(\xi_t)$ that has
\begin{equation}
\label{eq:restriction}
 \xi_t (X_t + j) = B \ \ \hbox{for all} \ j \geq 3
\end{equation}
 where
 $$ X_t := \max \,\{x \in \Z : \xi_t (y) = A \ \hbox{for all} \ y \leq x \} $$
 but otherwise evolves according to the evolution rules \eqref{eq:naming-game}.
 That is, the process starts from the configuration described in \eqref{eq:initial} and evolves according to the evolution
 rules of the naming game except that, each time the configuration violates condition \eqref{eq:restriction}, the state at
 vertex $X_t + 3$ instantaneously flips to a type $B$.
 In view of this new rule, Lemma \ref{lem:attractive} implies that
 $$ P \,(\xi_t (x) = A) \leq P \,(\zeta_t (x) = A) \quad \hbox{and} \quad P \,(\xi_t (x) = B) \geq P \,(\zeta_t (x) = B) $$
 for all $x \in \Z$ and $t > 0$, from which it follows that
 $$ \lim_{t \to \infty} P \,(\zeta_t (x) = A) = 1 \ \ \hbox{for all} \ x \in \Z $$
 whenever
\begin{equation}
\label{eq:coupling}
 \lim_{t \to \infty} X_t = \infty \ \ \hbox{almost surely}.
\end{equation}
 Moreover, one easily checks that the modified process $(\xi_t)$ only has three possible interfaces corresponding to the following
 three types of configurations:
 $$ \begin{array}{rl}
    \hbox{(type 0)} & \xi_t (X_t + j) = B \ \ \hbox{for all} \ j \geq 1 \vspace*{2pt} \\
    \hbox{(type 1)} & \xi_t (X_t + 1) = AB \ \ \hbox{and} \ \ \xi_t (X_t + j) = B \ \ \hbox{for all} \ j \geq 2 \vspace*{2pt} \\
    \hbox{(type 2)} & \xi_t (X_t + 1) = \xi_t (X_t + 2) = AB \ \ \hbox{and} \ \ \xi_t (X_t + j) = B \ \ \hbox{for all} \ j \geq 3. \end{array} $$
 Indeed, only the transitions $0 \to 1$ and $1 \to 0$ and $1 \to 2$ for the configuration types are allowed starting from a
 type 0 or a type 1 configuration.
 Moreover, from a type 2 configuration, either a monolingual and a bilingual individuals interact, which results in a type 1 configuration
 or a configuration with three bilingual individuals which instantaneously flips to a type 2 configuration, or  both bilingual individuals
 interact, which results in a type 0 configuration.
 The main reason for introducing this modified process is its mathematical tractability due to the small size of the state space of the
 process seen from the interface.
 As previously mentioned, this is further motivated by the fact that numerical simulations suggest that the naming game itself, when
 starting from configuration \eqref{eq:initial}, is most of the time in type 0, 1 or 2 configurations, so the analysis of the modified
 process allows to obtain a bound $c$ somewhat close to one.
 To establish Theorem \ref{th:1D}, we now prove that, under the conditions of the theorem, \eqref{eq:coupling} holds.
 This is done by first computing the occupation time of the modified process in each configuration type and then computing the value
 of the drift for the process $(X_t)$ in each configuration type.
 To shorten the notations as in the proof of Lemma \ref{lem:supermartingale}, we will again use the probabilities $q_A$ and $q_B$ defined
 in \eqref{eq:short}.

\begin{figure}[t]
\centering
\scalebox{0.45}{\input{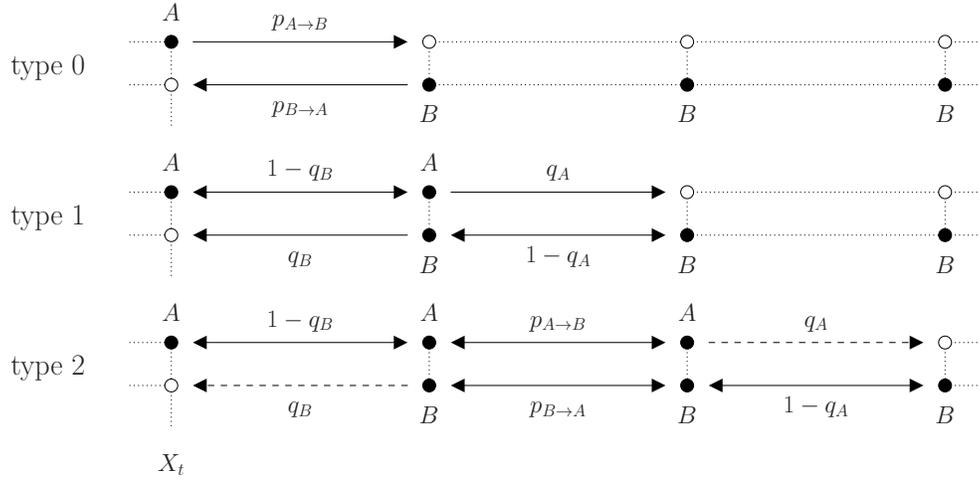}}
\caption{\upshape{Picture of the configuration types and all the possible transitions along with their rates.
 Each configuration is represented by two copies of the lattice, with the upper layer having a black particle to indicate that the individual
 uses word $A$ and a white particle when she does not, and similarly for word $B$ at the lower layer.
 Arrows indicate transitions where the individual at the tail speaks to the individual at the tip, while double arrows indicate transitions
 where any of the two neighbors speaks to the other one.
 The two dashed arrows in type 2 configurations correspond to the two transitions that are instantaneously followed by the event that the
 rightmost bilingual individual spontaneously becomes a type $B$ monolingual individual.}}
\label{fig:drift}
\end{figure}

\begin{lemma} --
\label{lem:config}
 The limits \ $\pi_j := \lim_{t \to \infty} \,P \,(\xi_t \ \hbox{is of type} \ j)$ exist and satisfy
\begin{equation}
\label{eq:config-1}
 \pi_0 \ = \ 2 \,\pi_1 + (r - 2) \,\pi_2 \quad \hbox{and} \quad r \,\pi_1 \ = \ (3 - r) \,\pi_2 \qquad \hbox{where} \quad r := q_A + q_B.
\end{equation}
\end{lemma}
\begin{proof}
 Let $Y_t := j$ if the configuration at time $t$ is of type $j$.
 Looking at all the possible updates of the modified naming game and the corresponding transition rates in Figure \ref{fig:drift}, one
 easily checks that the configuration type evolves according to the Markov chain with transitions
\begin{equation}
\label{eq:config-2}
  \begin{array}{ccl}
   0 \ \to \ 1 & \hbox{at rate} & r_{01} \ = \ p_{A \to B} + p_{B \to A} \ = \ 1 \vspace*{2pt} \\
   1 \ \to \ 0 & \hbox{at rate} & r_{10} \ = \ (1 - q_B) + (1 - q_A) \ = \ 2 - (q_A + q_B)\vspace*{2pt} \\
   1 \ \to \ 2 & \hbox{at rate} & r_{12} \ = \ q_A + q_B \vspace*{2pt} \\
   2 \ \to \ 0 & \hbox{at rate} & r_{20} \ = \ p_{A \to B} + p_{B \to A} \ = \ 1 \vspace*{2pt} \\
   2 \ \to \ 1 & \hbox{at rate} & r_{21} \ = \ (1 - q_B) + (1 - q_A) \ = \ 2 - (q_A + q_B). \end{array}
\end{equation}
 Note that the rates on the two dashed arrows of Figure \ref{fig:drift} are irrelevant.
 These transition rates imply that $(Y_t)$ is irreducible therefore the limits
 $$ \pi_j \ := \ \lim_{t \to \infty} \,P \,(\hbox{configuration} \ \xi_t \ \hbox{is of type} \ j) \ = \ \lim_{t \to \infty} \,P \,(Y_t = j) \quad \hbox{for} \ j = 0, 1, 2 $$
 exist and satisfy the following two equations:
 $$ \pi_0 \ = \ (r_{10} + r_{12}) \,\pi_1 + r_{21} \,\pi_2 \quad \hbox{and} \quad r_{12} \,\pi_1 = (1 + r_{21}) \,\pi_2. $$
 Using also that $r_{12} = r$ and $r_{10} = r_{21} = 2 - r$ according to \eqref{eq:config-2} gives
 $$ \pi_0 \ = \ 2 \,\pi_1 + (2 - r) \,\pi_2 \quad \hbox{and} \quad r \,\pi_1 = (3 - r) \,\pi_2 $$
 which is precisely \eqref{eq:config-1}.
\end{proof} \\ \\
 To prove \eqref{eq:coupling}, the next step is to compute the value of the conditional drift of the process $(X_t)$ given the
 configuration type, i.e.,
 $$ \begin{array}{l} D_j \ := \ \lim_{\,h \to 0} \ h^{-1} \,E \,(X_{t + h} - X_t \,| \,Y_t = j) \quad \hbox{for} \ j = 0, 1, 2. \end{array} $$
 Looking again at all the possible updates, one easily finds
\begin{equation}
\label{eq:drift}
  \begin{array}{rcl}
   D_0 & = & - p_{B \to A} \vspace*{2pt} \\
   D_1 & = & (1 - q_B) - q_B \ = \ 1 - 2 \,q_B \vspace*{2pt} \\
   D_2 & = & (1 - q_B) - q_B + 2 \,p_{A \to B} \ = \ D_1 + 2 \,p_{A \to B}. \end{array}
\end{equation}
 The last step is to combine \eqref{eq:config-1} and \eqref{eq:drift} to prove that
\begin{equation}
\label{eq:global-drift}
 \pi_0 \,D_0 + \pi_1 D_1 + \pi_2 \,D_2 \ > \ 0 \quad \hbox{for all} \ \phi > c
\end{equation}
 from which the almost sure convergence of the interface to infinity follows.
 The most natural approach is to express $\pi_j$ and $D_j$ for $j = 0, 1, 2$ as a function of $\phi$, from which it
 can be deduced that the first inequality in \eqref{eq:global-drift} holds for $\phi$ larger than the largest real root of a certain polynomial
 with degree six.
 This root is not obvious to compute.
 Instead, we observe that, when both fitnesses are close to each other, $\phi$ is close to one and the rate $r$ close to $1/2$.
 The next two lemmas show that the left-hand side of \eqref{eq:global-drift} is larger than its counterpart obtained by computing
 $\pi_j$ under the assumption $r = 1/2$, which allows to express $c$ more simply as the largest root of a polynomial with degree two.
 Interestingly, a series of evaluations of the polynomial with degree six around $c$ indicates that the largest real root of this
 polynomial only differ from $c$ by less than $10^{-6}$, which shows \emph{a posteriori} the advantage of our approach.

\begin{lemma} --
\label{lem:rate}
 For all $\phi_A, \phi_B > 0$, we have $1/2 \leq r \leq 1$.
\end{lemma}
\begin{proof}
 Recalling \eqref{eq:config-1} and using $\phi_A + \phi_B = 2 \,\phi_{AB}$, we get
 $$ \begin{array}{rcl}
     r & := & q_A + q_B \vspace*{4pt} \\
       &  = & p_{AB \to A} \ p_{B \to A} + p_{AB \to B} \ p_{A \to B} \vspace*{4pt} \\
       &  = & \phi_B \,(2 \,\phi_A + 2 \,\phi_{AB})^{-1} + \phi_A \,(2 \,\phi_B + 2 \,\phi_{AB})^{-1} \vspace*{4pt} \\
       &  = & \phi_B \,(3 \,\phi_A + \phi_B)^{-1} + \phi_A \,(3 \,\phi_B + \phi_A)^{-1}
       \  = \ (3 \phi + 1)^{-1} + (3 \phi^{-1} + 1)^{-1} \ =: \ h (\phi). \end{array} $$
 Noticing that $h (\phi) = h (\phi^{-1})$ and differentiating with respect to $\phi$, we deduce that
 $$ h (1) \ = \ 1/2 \ \leq \ r \ \leq \ 1 \ = \ \lim_{\phi \to \infty} h (\phi), $$
 which completes the proof.
\end{proof}

\begin{lemma} --
\label{lem:sign-drift}
 For all $\phi_A, \phi_B > 0$, we have
\begin{equation}
\label{eq:sign-drift-1}
 \sgn (\pi_0 \,D_0 + \pi_1 D_1 + \pi_2 \,D_2) \ \geq \ \sgn (17 \,D_0 + 10 \,D_1 + 2 \,D_2).
\end{equation}
\end{lemma}
\begin{proof}
 Using the relationship among $\pi_0, \pi_1$ and $\pi_2$ given in \eqref{eq:config-1}, we obtain
 $$ \begin{array}{rcl}
    \sgn (\pi_0 \,D_0 + \pi_1 D_1 + \pi_2 \,D_2) & = &
    \sgn ((2 \,\pi_1 + (r - 2) \,\pi_2) \,D_0 + \pi_1 \,D_1 + \pi_2 \,D_2) \vspace*{4pt} \\ & = &
    \sgn ((2 D_0 + D_1) \,\pi_1 + ((r - 2) \,D_0 + D_2) \,\pi_2) \vspace*{4pt} \\ & = &
    \sgn ((2 D_0 + D_1) (3 - r) + ((r - 2) \,D_0 + D_2) \,r). \end{array} $$
 To find a lower bound for the sign above, we introduce the function
 $$ D (r) \ := \ (2 D_0 + D_1) (3 - r) + ((r - 2) \,D_0 + D_2) \,r $$
 and observe that, for all $r \leq 1$,
 $$ \begin{array}{rcl}
     D' (r) & = & - (2 D_0 + D_1) + ((r - 2) \,D_0 + D_2) + D_0 \,r \vspace*{4pt} \\
            & = & 2 \,(r - 2) \,D_0 - D_1 + D_2
            \ = \ - 2 \,(r - 2)(1 - p_{A \to B}) + 2 \,p_{A \to B} \vspace*{4pt} \\
            & = & - 2 \,(r - 2) + 2 \,(r - 1) \,p_{A \to B}
            \ \geq \ - 2 \,(r - 2) + 2 \,(r - 1) \ = \ 2 \ > \ 0. \end{array} $$
 Using in addition that $1/2 \leq r \leq 1$ according to Lemma \ref{lem:rate} gives
 $$ \begin{array}{l}
    \sgn (\pi_0 \,D_0 + \pi_1 D_1 + \pi_2 \,D_2) \ \geq \ \sgn (D (1/2)) \vspace*{4pt} \\ \hspace*{40pt} = \
    \sgn ((2 D_0 + D_1) (3 - 1/2) + ((1/2 - 2) \,D_0 + D_2) \,(1/2)) \vspace*{4pt} \\ \hspace*{40pt} = \
    \sgn (17 \,D_0 + 10 \,D_1 + 2 \,D_2). \end{array} $$
 This completes the proof.
\end{proof}

\begin{lemma} --
\label{lem:positive-drift}
 The right-hand side of \eqref{eq:sign-drift-1} is positive whenever
\begin{equation}
\label{eq:positive-drift-1}
 \phi \ > \ c \quad \hbox{where} \quad c \ := \ \frac{23 + \sqrt{6097}}{96} \ \approx \ 1.053.
\end{equation}
\end{lemma}
\begin{proof}
 First of all, note that
 $$ p_{B \to A} \ = \ (\phi + 1)^{-1} \qquad q_A \ = \ \phi \,(\phi + 3)^{-1} \qquad q_B \ = \ (3 \phi + 1)^{-1}. $$
 Using in addition \eqref{eq:drift} gives
 $$ \begin{array}{rcl}
     F_0 & := & (\phi + 1)(3 \phi + 1) \,D_0 \ = \ - (3 \phi + 1) \vspace*{2pt} \\
     F_1 & := & (\phi + 1)(3 \phi + 1) \,D_1 \ = \   (3 \phi - 1)(\phi + 1)\vspace*{2pt} \\
     F_2 & := & (\phi + 1)(3 \phi + 1) \,D_2 \ = \   (3 \phi - 1)(\phi + 1) + 2 \phi \,(3 \phi + 1). \end{array} $$
 Since $(\phi + 1)(3 \phi + 1) > 0$, we deduce that
 $$ \begin{array}{l}
    \sgn (17 \,D_0 + 10 \,D_1 + 2 \,D_2) \ = \
    \sgn (17 \,F_0 + 10 \,F_1 + 2 \,F_2) \vspace*{4pt} \\ \hspace*{40pt} = \
    \sgn (- 17 \,(3 \phi + 1) + 12 \,(3 \phi - 1)(\phi + 1) + 4 \phi \,(3 \phi + 1)) \vspace*{4pt} \\ \hspace*{40pt} = \
    \sgn (48 \phi^2 - 23 \phi - 29) \end{array} $$
 which is positive whenever $\phi > c$ as defined in \eqref{eq:positive-drift-1}.
\end{proof} \\ \\
 From Lemma \ref{lem:positive-drift}, it directly follows that the process $(X_t)$ converges almost surely to infinity,
 which also implies convergence of the naming game starting from configuration \eqref{eq:initial} to the configuration
 in which all individuals are type $A$ monolingual.
 Moreover, we have
 $$ P \,(X_t \geq 0 \ \hbox{for all} \ t) \ > \ 0. $$
 To deduce that word $A$ can invade word $B$, we let $(X_t^+)$ and $(X_t^-)$ be two independent copies of the process $(X_t)$ and use
 a standard coupling argument to conclude that, under the assumptions of the theorem, the probability that the naming game starting
 with the origin in state $A$ and all the other vertices in state $B$ converges to the ``all $A$'' configuration is given by
 $$ \begin{array}{rcl}
     P \,(X_t^+ \geq - X_t^- \ \hbox{for all} \ t) & \geq &
     P \,(X_t^+ \geq 0 \ \hbox{and} \ - X_t^- \leq 0 \ \hbox{for all} \ t) \vspace*{4pt} \\ & \geq &
     P \,(X_t^+ \geq 0 \ \hbox{for all} \ t) \ P \,(- X_t^- \leq 0 \ \hbox{for all} \ t) \vspace*{4pt} \\ & \geq &
     P \,(X_t \geq 0 \ \hbox{for all} \ t) \ P \,(X_t \geq 0 \ \hbox{for all} \ t) \ > \ 0. \end{array} $$
 Since there is a positive probability for the process starting with a single bilingual individual at the origin that the origin
 is of type $A$ at time one, this completes the proof of Theorem \ref{th:1D}.


\section{The naming game in higher dimensions}
\label{sec:2D}

\indent This section is devoted to proving Theorem \ref{th:lattice}, which relies on a block construction.
 To spare the reader complicated notations, we only prove the result in $d = 2$ but our approach easily extends to higher
 dimensions.
 The idea of the block construction is to couple a certain collection of good events related to the process properly
 rescaled in space and time with the set of open sites of oriented site percolation on the oriented graph $\mathcal H_1$
 with vertex set
 $$ H \ := \ \{(z, n) \in \Z^2 \times \Z_+ : z_1 + z_2 + n \ \hbox{is even} \} $$
 and in which there is an oriented edge
 $$ (z, n) \to (z', n') \quad \hbox{if and only if} \quad z' = z + (\pm 1, \pm 1) \ \ \hbox{and} \ \ n' = n + 1. $$
 See the left-hand side of Figure \ref{fig:graphs} for a picture in $d = 1$.
 To rescale the process and define the collection of good events later in the proof of Lemma \ref{lem:block}, we let
 $T := \sqrt{\phi}$ and introduce the collection of space-time blocks
\begin{equation}
\label{eq:blocks}
  \begin{array}{l}
     B (z, n) \ := \ \{(x, t) = ((x_1, x_2), t) \in \Z^2 \times [0, \infty) \ \hbox{such that} \vspace*{4pt} \\ \hspace*{15pt}
                        x_j \in \{z_j, z_j + 1 \} \ \hbox{for} \ j = 1, 2 \ \hbox{and} \ t \in [2nT, 2 (n + 1) \,T) \} \ \ \hbox{for all} \ \ (z, n) \in H. \end{array}
\end{equation}
 In words, space is partitioned into $2 \times 2$ squares and time into intervals of length $2T$, while the collection
 of space-time blocks in \eqref{eq:blocks} defines a partition of the space-time universe.
 The key to proving invasion of word $A$ is to show that the set of sites
 $$ (z, n) \in H \quad \hbox{such that} \quad \eta_t (x) = A \ \ \hbox{for all} \ \ (x, t) \in B (z, n), $$
 that we call $A$-sites for short, dominates stochastically the set of wet sites in an oriented site percolation process whose
 parameter can be made arbitrarily close to one by choosing the parameter~$\phi$ sufficiently large.
 More precisely, we have the following lemma.

\begin{lemma} --
\label{lem:block}
 For all $\ep > 0$, there exists $\phi > 0$ such that the set of $A$-sites dominates the set of wet sites in a
 two dependent oriented site percolation process with parameter $1 - \ep$.
\end{lemma}
\begin{proof}
 We say that the interaction along edge $(x, y)$ at time $T_n (x, y)$ is
 $$ \hbox{a good interaction if} \quad U_n (x, y) \ < \ q_A \ = \ \phi \,(\phi + 3)^{-1} $$
 and a bad interaction otherwise.
 Referring to Figure \ref{fig:block}, we let $G (z, n)$ be the event that
\begin{enumerate}
 \item between time $2nT$ and time $(2n + 1) \,T$, there are at least two good interactions along each of the eight edges labeled 1
  on the left-hand side, \vspace*{4pt}
 \item between time $2nT$ and time $(2n + 1) \,T$, there is no bad interaction along any of the sixteen edges labeled 2 on the left-hand
  side, \vspace*{4pt}
 \item between time $(2n + 1) \,T$ and time $2 (n + 1) \,T$, there is at least one good and no bad interaction along each
  of the eight edges labeled 3 on the right-hand side, \vspace*{4pt}
 \item between time $(2n + 1) \,T$ and time $(2n + 4) \,T$, there is no bad interaction along any of the sixteen edges labeled 4 on
  the right-hand side.
\end{enumerate}
\begin{figure}[t]
\centering
\scalebox{0.40}{\input{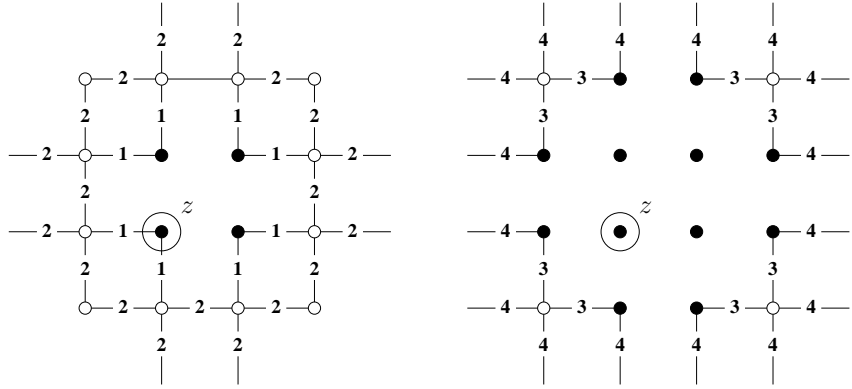}}
\caption{\upshape{Picture of the block construction.}}
\label{fig:block}
\end{figure}
 From \eqref{eq:attractive-1}--\eqref{eq:attractive-2} and the probabilities in Table \ref{tab:coupling}, it follows that an interaction
 involving at least one individual using word $A$ can only result in one of the transitions 1A--5A in the table.
 In particular, whenever site $(z, n)$ is an $A$-site and our good event 1--4 occurs, the following holds:
\begin{itemize}
 \item at time $(2n + 1) \,T$, all twelve vertices marked with a black dot $\bullet$ on the right-hand side of the figure are of
  type $A$ and \vspace*{4pt}
 \item between $(2n + 2) \,T$ and $(2n + 4) \,T$, all sixteen vertices in the figure are of type $A$.
\end{itemize}
 In particular, letting $\Omega (z, n)$ be the event that $(z, n)$ is an $A$-site, we deduce that
\begin{equation}
\label{eq:block-1}
  \Omega (z, n) \,\cap \,G (z, n) \ \subset \ \Omega ((z_1 \pm 1, z_2 \pm 1), n + 1).
\end{equation}
 Now, let $X$ and $Y$ be the number of good and bad interactions that occur along one given edge in a given time interval of length $T$.
 Since interactions occur along each edge of the lattice at rate one and are independently good with probability $\phi \,(\phi + 3)^{-1}$
 $$ X = \poisson (\phi T \,(\phi + 3)^{-1}) \quad \hbox{and} \quad Y = \poisson (3 T \,(\phi + 3)^{-1}). $$
 In particular, for all $\ep > 0$, the probability of the good event 1--4 is
\begin{equation}
\label{eq:block-2}
  \begin{array}{l}
     P \,(G (z, n)) \ \geq \
     1 - 8 \,P \,(X < 2) - 16 \,P \,(Y \neq 0) \vspace*{4pt} \\ \hspace*{100pt} - 8 \,P \,(X = 0) - 8 \,P \,(Y \neq 0) - 16 \times 3 \,P \,(Y \neq 0) \vspace*{4pt} \\ \hspace*{25pt} = \
     1 - 16 \,P \,(X = 0) - 8 \,P \,(X = 1) - 72 \,P \,(Y \neq 0) \vspace*{4pt} \\ \hspace*{25pt} = \
     1 - 8 \,(2 + \phi T \,(\phi + 3)^{-1}) \,\exp (- \phi T \,(\phi + 3)^{-1}) - 72 \,(1 - \exp \,(- 3 T \,(\phi + 3)^{-1})) \vspace*{4pt} \\ \hspace*{25pt} \geq \
     1 - 8 \,(2 + \phi T \,(\phi + 3)^{-1}) \,\exp (- \phi T \,(\phi + 3)^{-1}) - 216 \,T \,(\phi + 3)^{-1} \vspace*{4pt} \\ \hspace*{25pt} = \
     1 - 8 \,(2 + \phi \,\sqrt{\phi} \,(\phi + 3)^{-1}) \,\exp (- \phi \,\sqrt{\phi} \,(\phi + 3)^{-1}) - 216 \,\sqrt{\phi} \,(\phi + 3)^{-1} \ \geq \ 1 - \ep \end{array}
\end{equation}
 for all $\phi$ large enough.
 Finally, we observe that the good event $G (z, n)$ is measurable with respect to the graphical representation in the space-time region
 $$ (z, 2nT) + \{[-2, 3] \times [0, 4T) \} \ \subset \ \Z^2 \times [0, \infty). $$
 This, together with the inclusion \eqref{eq:block-1} and the lower bound \eqref{eq:block-2} are exactly the comparison assumptions
 of Theorem 4.3 in \cite{durrett_1995}, from which the lemma directly follows.
\end{proof} \\ \\
 It is known from standard results based on the so-called contour argument that, for $\ep > 0$ small enough, there exists
 with positive probability an infinite cluster of wet sites in the two dependent oriented site percolation process on $\mathcal H_1$
 starting with one open site at level 0 and in which sites at the other levels are open with probability $1 - \ep$.
 This, together with Lemma \ref{lem:block}, implies that, for the naming game starting with a single bilingual individual,
 $$ \liminf_{t \to \infty} \ P \,(\eta_t (x) = A) \ > \ 0 \quad \hbox{for all} \ x \in \Z^2. $$
 This proves survival of word $A$ but not extinction of word $B$ with positive probability.
 In fact, a weaker form of survival can be proved in the more general case when $\phi > 3$ by simply using techniques similar to the
 ones in the proof of Lemma \ref{lem:supermartingale} to show that the number of individuals using word $A$ is a submartingale.
 However, extinction of word $B$ with positive probability cannot be deduced from this approach.
 In contrast, our coupling with oriented site percolation combined with an idea of the author \cite{lanchier_2012} that extends a result
 of Durrett \cite{durrett_1992} can be used to complete the proof of the theorem.
 This is done in the next lemma.

\begin{lemma} --
\label{lem:dry}
 For all $\phi$ large enough we have $p_A > 0$.
\end{lemma}
\begin{proof}
 Throughout the proof, we think of the naming game as being coupled with oriented site percolation as in the statement of Lemma \ref{lem:block}.
 To begin with, we follow \cite{lanchier_2012} by introducing the new oriented graph $\mathcal H_2$ with the same vertex set as $\mathcal H_1$
 but in which there is an oriented edge
 $$ \begin{array}{rcl}
     (z, n) \to (z', n') & \hbox{if and only if} & (z' = z + (\pm 1, \pm 1) \ \ \hbox{and} \ \ n' = n + 1) \vspace*{2pt} \\ && \hspace*{15pt}
                                     \hbox{or} \ \ (z' = z + (\pm 2, \pm 2) \ \ \hbox{and} \ \ n' = n). \end{array} $$
 See the right-hand side of Figure \ref{fig:graphs} for a picture in $d = 1$.
\begin{figure}[t]
\centering
\scalebox{0.40}{\input{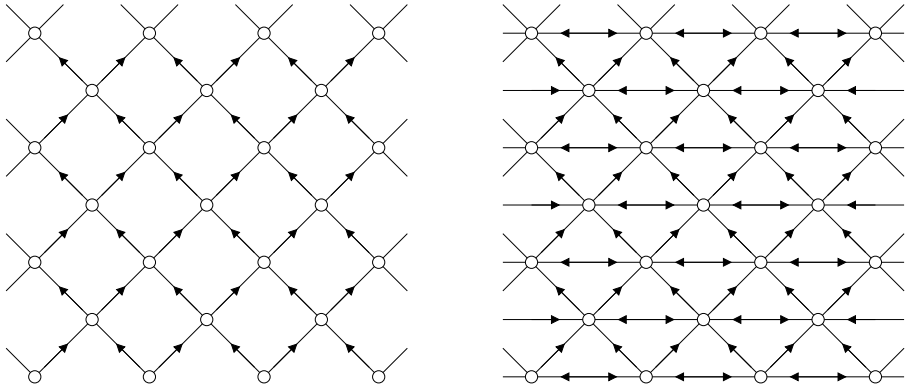}}
\caption{\upshape{Picture of the graphs $\mathcal H_1$ and $\mathcal H_2$ in dimension $d = 1$.}}
\label{fig:graphs}
\end{figure}
 We say that a site in the percolation process is dry if it is not wet.
 Also, for $j = 1, 2$, we write $(w, 0) \to_j (z, n)$ and say that there is a dry path connecting both sites if there is a sequence
 $$ (z_0, 0) = (w, 0), \ (z_1, n_1), \ \ldots, \ (z_k, n_k) = (z, n) \in H $$
 such that the following two conditions hold:
\begin{enumerate}
 \item $(z_i, n_i) \to (z_{i + 1}, n_{i + 1})$ is an oriented edge in $\mathcal H_j$ for all $i = 0, 1, \ldots, k - 1$ and \vspace{4pt}
 \item the site $(z_i, n_i)$ is dry for all $i = 0, 1, \ldots, k$.
\end{enumerate}
 Note that a dry path in $\mathcal H_1$ is also a dry path in $\mathcal H_2$ but the reciprocal is false.
 Now, the proofs of Lemmas 4--11 in Durrett \cite{durrett_1992} imply the following:
 there exists $\ep > 0$ small such that, for the percolation process on $\mathcal H_1$ with parameter $1 - \ep$ starting with $(0, 0)$
 open and all the other sites closed at level zero, conditioned on the event that percolation occurs, we have
\begin{equation}
\label{eq:dry-1}
  \begin{array}{l}
    \displaystyle \lim_{m \to \infty} \ P \,((w, 0) \to_1 (z, n) \ \hbox{for some $w \in 2 \Z^2$,} \vspace*{0pt} \\ \hspace*{100pt}
    \hbox{some $z \in B_2 (0, na)$ and some $n \geq m$}) \ = \ 0 \end{array}
\end{equation}
 for some $a > 0$.
 In words, if the density of open sites is close enough to one, there is a linearly expanding region in which (even closed) sites
 cannot be reached from a path of dry sites starting at level zero.
 This applies to dry paths in the graph $\mathcal H_1$ but as pointed out in \cite{lanchier_2012}, the proofs of Lemmas 4--11
 in Durrett \cite{durrett_1992} easily extend to give \eqref{eq:dry-1} for dry paths in $\mathcal H_2$.
 To conclude the proof, the last step is to show the connection between dry paths and $A$-sites.
 Assume that
\begin{equation}
\label{eq:event-1}
  \eta_t (x) \neq A \quad \hbox{for some} \quad x \in \Z^2 \ \ \hbox{and} \ \ t \in [2nT, 2 (n + 1) \,T).
\end{equation}
 Since word $B$ cannot appear spontaneously, this implies the existence of
 $$ \begin{array}{l}
     x_0, x_1, \ldots, x_m = x \in \Z^2 \quad \hbox{and} \quad s_0 = 0 < s_1 < \cdots < s_{m + 1} = t \vspace*{4pt} \\ \hspace*{50pt}
    \hbox{such that} \quad \eta_s (x_j) \neq A \ \ \hbox{for all} \ \ s_j \leq s \leq s_{j + 1} \ \hbox{and} \ j = 0, 1, \ldots, m, \end{array} $$
 which in turn implies that
\begin{equation}
\label{eq:event-2}
  (w, 0) \to_2 (z, n) \ \ \hbox{for some} \ \ w \in 2 \Z^2 \ \ \hbox{and} \ \ (z, n) \ \hbox{such that} \ (x, t) \in B (z, n).
\end{equation}
 Note however that this does not imply the existence of a dry path in $\mathcal H_1$ which is the reason why we introduced a new
 graph with additional edges.
 Taking the probability of the event in \eqref{eq:event-2} and the probability of the sub-event in \eqref{eq:event-1} directly gives
\begin{equation}
\label{eq:dry-2}
 P \,(\eta_t (x) \neq A) \ \leq \ P \,((w, 0) \to_2 (z, n) \ \hbox{for some} \ w \in 2 \Z^2)
\end{equation}
 where $(z, n)$ is the unique site such that $(x, t) \in B (z, n)$.
 Since $\ep > 0$ can be made arbitrarily small by choosing $\phi$ large, the analog of \eqref{eq:dry-1} for oriented dry paths in the
 graph $\mathcal H_2$ together with the inequality \eqref{eq:dry-2} implies that, conditioned on the event that percolation occurs,
 $$ \lim_{t \to \infty} \ P \,(\eta_t (x) = A) \ = \ 1 \quad \hbox{for all} \ x \in \Z^2 $$
 for the naming game conditioned on the event that $(0, 0)$ is an $A$-site.
 Since the probability that percolation occurs is positive for $\ep > 0$ small and since there is a positive probability for the
 process starting with a single bilingual individual at the origin that all sites in the spatial box $\{0, 1 \}^d$ are of type $A$
 at time one, the lemma and Theorem \ref{th:lattice} follow.
\end{proof}


\end{document}